\documentclass[10pt]{amsart}
\usepackage{amsthm}
\usepackage{graphicx}
\usepackage{amssymb}
\usepackage{color}
\usepackage{tikz}
\usepackage{amsmath}
\usepackage{hyperref}
\usepackage{float}
\usepackage{enumerate}

\DeclareMathAlphabet\mathbfcal{OMS}{cmsy}{b}{n}

\numberwithin{equation}{section}

\newtheorem{theorem}{Theorem}[section]
\newtheorem{definition}[theorem]{Definition}

\newtheorem{lemma}[theorem]{Lemma}

\newtheorem{proposition}[theorem]{Proposition}
\newtheorem{corollary}[theorem]{Corollary}
\newtheorem{remark}[theorem]{Remark}

\newcommand{\PP}{{\mathbb P}} \newcommand{\RR}{\mathbb{R}}
\newcommand{\QQ}{\mathbb{Q}} \newcommand{\CC}{{\mathbb C}}
\newcommand{\ZZ}{{\mathbb Z}} \newcommand{\NN}{{\mathbb N}}

\newcommand{\cF}{\mathcal{F}}

\newcommand{\cE}{\mathcal{E}}

\newcommand{\cH}{\mathcal{H}}

\newcommand{\cO}{\mathcal{O}}

\newcommand{\fa}{\mathfrak{a}}
\newcommand{\fb}{\mathfrak{b}}

\newcommand{\vol}{\operatorname{vol}}
\newcommand{\ord}{\mathrm{ord}}

\newcommand{\lct}{\mathrm{lct}}

\newcommand{\Val}{\mathrm{Val}}

\title{Valuative invariants with higher moments}

\usepackage{hyperref}
\hypersetup{colorlinks,linkcolor={blue},citecolor={red}}

\begin{document}

\author[K. Zhang]{Kewei Zhang}
\address{Laboratory of Mathematics and Complex Systems, School of Mathematical Sciences, Beijing Normal University, Beijing, 100875, China.}
\email{kwzhang@pku.edu.cn}

\maketitle

\begin{abstract}
In this article we introduce a family of valuative invariants defined in terms of the $p$-th moment of the expected vanishing order. These invariants lie between $\alpha$ and $\delta$-invariants. They vary continuously in the big cone and semi-continuously in families. Most importantly, they give sufficient conditions for K-stability of Fano varieties, which generalizes the $\alpha$ and $\delta$-criterions in the literature. They are also related to the $d_p$-geometry of maximal geodesic rays.
\end{abstract}

\tableofcontents

\section{Introduction}

\subsection{Background}
Valuative invariants play significant roles in finding canonical metrics on polarized varieties. A highly notable one is the $\alpha$-invariant that goes back to Tian \cite{Tian87}, with the help of which there is now a huge table of K\"ahler--Einstein Fano varieties that have been discovered by various authors. Another remarkable invariant is the $\delta$-invariant that was recently introduced by Fujita-Odaka \cite{FO18}, which turns out to be a very powerful new tool in the study of K-stability of Fano varieties and now there is a large literature on it; see especially the work of Blum--Jonsson 
\cite{BJ17} and the references therein. The purpose of this article is to further polish the pertinent field by introducing a family of valuative invariants that interpolates betweent $\alpha$ and $\delta$. As we shall see, these invariants enjoy many properties that were previously only established for $\alpha$ and $\delta$. Our work suggests that the study of valuative invariants can be carried out in a broader context and hopefully this will serve as a new perspective in future research.

To begin with, we first recall the definition of $\delta$. Let $X$ be an $n$-dimensional complex normal projective variety with at worst klt singularities and $L$ is an ample line bundle. Up to a multiple of $L$, we assume throughout that $H^0(X,mL)\neq0$
for any $m\in\ZZ_{>0}$. Put
$$
d_m:=h^0(X,mL),\ m\in\NN.
$$
Consider a basis $s_1,\cdots, s_{d_m}$ of the vector space $H^0(X,mL)$,
which induces an effective $\mathbb{Q}$-divisor
$$
D:=\frac{1}{md_m}\sum_{i=1}^{d_m}\big\{s_i=0\big\}\sim_{\mathbb{Q}}L.
$$
Any $\mathbb{Q}$-divisor $D$ obtained in this way is called an $m$-basis type divisor of $L$. Let
$$
\delta_m\big(L\big):=\inf\bigg\{\lct(X,D)\bigg|D\text{ is $m$-basis type of }L\bigg\}.
$$
Then let
$$
\delta(L)=\limsup_{m\rightarrow\infty}\delta_m(L).
$$
This limsup is in fact a limit by \cite{BJ17}.
So roughly speaking, $\delta(X,L)$ measures the singularities of basis type divisors of $L$.


The following result demonstrates the importance of the $\delta$-invariant.
\begin{theorem}[\cite{FO18,BJ17}]
\label{theorem:delta}
Let $X$ be a $\QQ$-Fano variety.
The following assertions hold:
\begin{enumerate}
\item $X$ is $K$-semistable if and only if $\delta(-K_X)\geqslant1$;
\item $X$ is uniformly $K$-stable if and only if $\delta(-K_X)>1$.
\end{enumerate}
\end{theorem}

Thus by \cite{BBJ18,LTW19}, $\delta$-invariant serves as a criterion for the existence of KE metrics on $\QQ$-Fano varieties.

\subsection{Valuative characterization}

Let $\pi:Y\rightarrow X$ be a proper birational morphism from a normal variety $Y$ to $X$
and let $F\subset Y$ be a prime divisor $F$ in $Y$.
Such an $F$ will be called a divisor \emph{over} $X$. Let
$$
S_m(L,F):=\frac{1}{md_m}\sum_{j=1}^{\tau_m(L,F)}\dim H^0(Y,m\pi^* L-jF)
$$
denote the \emph{expected vanishing order of $L$ along $F$ at level $m$}.
Here 
$$
\tau_m(L,F):=\sup_{0\neq s\in H^0(X,mL)}\ord_F(s)
$$
denotes the \emph{pseudo-effective threshold of $L$ along $F$ at level $m$}.
Then a basic but important linear algebra lemma due to Fujita--Odaka \cite{FO18} says that
$$
S_m(L,F)=\sup\bigg\{\ord_F(D)\,\bigg|\,m\text{-basis divisor $D$ of }L\bigg\},
$$
and this supremum is attained by any $m$-basis divisor $D$ arising from a basis $\{s_i\}$ that is \emph{compatible} with the filtration
$$
H^0\big(Y,m\pi^*L\big)\supset H^0\big(Y,m\pi^*L-F\big)\supset\cdots\supset H^0\big(Y,m\pi^*L-(\tau_m(L,F)+1)F\big)=\{0\},
$$
meaning that each $H^0(Y,m\pi^* L-jF)$
is spanned by a subset of the $\{s_i\}_{i=1}^{d_m}$. 
Then it is easy to deduce that
$$
	\delta_m(L)=\inf_{F}\frac{A_X(F)}{S_m(L,F)}.
$$

As $m\rightarrow\infty$, one has
$$
S(L,F):=\lim_{m\rightarrow\infty}S_m(L,F)=\frac{1}{\vol(L)}\int_0^{\tau(L,F)}\vol(\pi^*L-xF)dx,
$$
which is called the \emph{expected vanishing order of $L$ along $F$}. 
Then Blum--Jonsson \cite{BJ17} further show that the limit of $\delta_m(X,L)$ also exists, and is equal to
$$
\delta(L)=\inf_{F}\frac{A_X(F)}{S(L,F)}.
$$
Another closely related valuative invariant is Tian's $\alpha$-invariant \cite{Tian87}, which can be defined as (cf. \cite[Appendix]{CS08} and \cite[Theorem C]{BJ17})
$$
\alpha(L):=\inf_F\frac{A_X(F)}{\tau(L,F)},
$$
where $\tau(L,F):=\lim\tau_m(L,F)/m$, the \emph{pseudo-effective threshold of $L$ along $F$}.

An important property of $S(L,F)$ is illustrated by the following result of K. Fujita, who shows that $S(L,F)$ can be viewed as the coordinate of the barycenter of certain Newton--Okounkov body along the ``$F$-axis'', and hence the well-known Brunn--Minkovski inequality in convex geometry implies the following estimate.

\begin{proposition}[Barycenter inequality in \cite{Fuj19b}]
\label{prop:Fujita-barycenter-ineq}
For any $F$ over $X$, one has
$$
\frac{\tau(L,F)}{n+1}\leq S(L,F)\leq\frac{n\tau(L,F)}{n+1}.
$$
\end{proposition}

One should think of $\tau$ and $S$ as the non-Archimedean analogues of the $I$ and $I-J$ functionals of Aubin, and it is shown by Boucksom--Jonsson that the above result holds for general valuations as well (cf. \cite[Theorem 5.13]{BoJ18}).  
An immediate consequence of Proposition \ref{prop:Fujita-barycenter-ineq} is the following:
\begin{equation}
\label{eq:alpha=delta}
    \frac{n+1}{n}\alpha(L)\leq\delta(L)\leq(n+1)\alpha(L).
\end{equation}


\subsection{Valuative invariants with higher moments, and the main results}
$S(L,F)$ can be treated as the \emph{first moment} of the vanishing order of $L$ along $F$. In general for $p\geq1$, one can also consider the $p$-th moment of the vanishing order of $L$ along $F$. More precisely, given a basis $\{s_i\}$ of $H^0(X,mL)$ that is compatible with the filtration induced by $F$, put
$$
S^{(p)}_m(L,F):=\frac{1}{d_m}\sum_{i=1}^{d_m}\bigg(\frac{\ord_F(s_i)}{m}\bigg)^p.
$$
In the K-stability literature, related $L^p$ notions (for test configurations) have been introduced and studied by Donaldson \cite{Don05}, Dervan \cite{Der18}, Hismoto \cite{His16} et. al.
Our formulation of $S^{(p)}$ is inspired by the recent work of Han-Li \cite{HL20}, where a sequence of Monge--Amp\`ere energies with higher moments was considered.
We will give a more formal definition of this in Section \ref{sec:S-p}.
As $m\rightarrow\infty$, one has (see Lemma \ref{lem:S-p-formula})
$$
S^{(p)}(L,F):=\frac{1}{\vol(L)}\int_0^{\tau(L,F)}px^{p-1}\vol(L-xF)dx.
$$
So in particular
$
S(L,F)=S^{(1)}(L,F).
$
One main point of this article is to show that most properties established for $S$ in the literature hold for $S^{(p)}$ as well; see Section \ref{sec:S-p} for more details.

Extending Fujita's barycenter inequality to the $p$-th moment, we have

\begin{theorem}
\label{thm:p-barycenter-ineq}
Given any divisor $F$ over $X$, one has
$$
\frac{\Gamma(p+1)\Gamma(n+1)}{\Gamma(p+n+1)}\tau(L,F)^p\leq S^{(p)}(L,F)\leq\frac{n}{n+p}\tau(L,F)^p.
$$
Here $\Gamma(\cdot)$ is the gamma function.
\end{theorem}

Letting $p\rightarrow\infty$, one obtains the following
\begin{corollary}
\label{cor:tau=lim-S-p}
One has
$$
\tau(L,F)=\lim_{p\rightarrow\infty}S^{(p)}(L,F)^{1/p}.
$$
\end{corollary}

Relating to the $d_p$-geometry of maximal geodesic rays, $S^{(p)}(L,F)$ has the following pluri-potential interpretation (in Section \ref{sec:test-curve} we will prove a more general result that holds for linearly bounded filtrations). The proof relies on the main result in \cite{His16} and some non-Archimedean ingredients from \cite{BHJ17,BoJ18,BoJ18a}.

\begin{theorem}
\label{thm:S-p=d_p/t}
Let $\varphi_t^F$ be the maximal geodesic ray induced by $F$ (see Section \ref{sec:test-curve} for the setup and definition), then
$$
S^{(p)}(L,F)^{1/p}=\frac{d_p(0,\varphi^F_t)}{t}\text{ for all }t>0,
$$
where $d_p$ denotes the Finsler metric introduced by Darvas (see \eqref{eq:def-d-p}).
\end{theorem}

The set of moments $\{S^{(p)}\}$ can also be used to construct various kinds of valuative thresholds for $(X,L)$. It turns out that $\alpha$ and $\delta$ are only two special ones. To be more precise, one can put
\begin{equation*}
    \delta^{(p)}(L):=\delta^{(p)}(X,L):=\inf_F\frac{A_X(F)}{S^{(p)}(L,F)^{1/p}}.
\end{equation*}
As we shall see in Proposition \ref{prop:tilde-delta-p=delta-p}, here one can also take inf over all \emph{valuations}, which yields the same invariant.

It is interesting to note that $\{\delta^{(p)}(L)\}_{p\geq1}$ is a decreasing family of valuative invariants with
$$
\delta(L)=\delta^{(1)}(L)
\text{ and }
\alpha(L)=\lim_{p\rightarrow\infty}\delta^{(p)}(L).
$$ 
Moreover, observe that the valuative formulation of $\delta^{(p)}(L)$ also makes sense when $L$ is merely a big $\RR$-line bundle. We show that the continuity of $\delta$ established in \cite{Zhang20} holds for $\delta^{(p)}$ as well.

\begin{theorem}
\label{thm:cont}
$\delta^{(p)}(\cdot)$ is a continuous function on the big cone.
\end{theorem}


Furthermore, we have the following result, generalizing the work of Blum-Liu \cite{BL18b}.

\begin{theorem}
\label{thm:delta-m-p-lsc}
Let $\pi: X\rightarrow T$ is a projective family of varieties and $L$ is a $\pi$-ample Cartier divisor on $X$. Assume that $T$ is normal, $X_t$ is klt for all $t\in T$ and $K_{X/T}$ is $\QQ$-Cartier. Then the function
$$
T\ni t\mapsto\delta^{(p)}(X_t,L_t)
$$
is lower semi-continuous.
\end{theorem}



And also, in the Fano setting, it turns out that $\delta^{(p)}$ can give sufficient conditions for K-stability.

\begin{theorem}
\label{thm:delta-p=>KE}
Let $X$ be a $\QQ$-Fano variety of dimension $n$. If $$\delta^{(p)}(-K_X)>\frac{n}{n+1}\bigg(\frac{n+p}{n}\bigg)^{1/p},$$
then $X$ is uniformly K-stable and hence admits a K\"ahler--Einstien metric.
\end{theorem}

For $p=1$ this is simply the $\delta$-criterion of Fujita-Odaka \cite{FO18} (which says that $\delta(-K_X)>1$ implies uniform K-stability), while for $p=\infty$ this recovers the $\alpha$-criterion of Tian \cite{Tian87} and Odaka--Sano \cite{OS12} (which says that $\alpha(-K_X)>\frac{n}{n+1}$ implies uniform K-stability). Thus Theorem \ref{thm:delta-p=>KE} provides a bridge between $\alpha$ and $\delta$-invariants and gives rise to a family of valuative criterions for the existence of K\"ahler--Einstein metrics.

To show Theorem \ref{thm:delta-p=>KE}, the key new ingredient is the following monotonicity:
$$
\bigg(\frac{n}{n+p}\bigg)^{1/p}\delta^{(p)}(L) \text{ is non-increasing in }p.
$$
This is deduced from Proposition \ref{prop:H-p-increase}, which relies on the Brunn--Minkowski inequality and yields a new result in convex geometry regarding the $p$-th barycenter of a convex body that is probably of independent interest.

Interestingly, the above monotonicity is actually ``sharp", from which we can characterize the borderline case in Theorem \ref{thm:delta-p=>KE}.
\begin{theorem}
\label{thm:delta-p-equality}
If an $n$-dimensional Fano manifold
X satisfies that
$$\delta^{(p)}(-K_X)=\frac{n}{n+1}\bigg(\frac{n+p}{n}\bigg)^{1/p}$$
for some $p\in(1,\infty]$. Then either $X=\PP^1$ or $X$ is K-stable.
In particular, X admits a K\"ahler-Einstein metric.
\end{theorem}

When $p=\infty$ this is exactly the main theorem of Fujita \cite{Fuj19c} (which says that $\alpha(-K_X)=\frac{n}{n+1}$ implies KE when $X$ is smooth). The proof of Theorem \ref{thm:delta-p-equality} uses the strategy of \cite{Fuj19c}, but one major difference is that the equality in our 
case is more difficult to grasp, which requires some subtle measure-theoretic argument. 

\textbf{Organization.}
The rest of this article is organized as follows.
In Section \ref{sec:S-p} we introduce $S^{(p)}$ in a more formal way, using filtrations. In Section \ref{sec:conti} we prove Theorem \ref{thm:cont} and in Section \ref{sec:lsc} we prove Theorem \ref{thm:delta-m-p-lsc}. Then Theorem \ref{thm:p-barycenter-ineq}, Theorem \ref{thm:delta-p=>KE} and Theorem \ref{thm:delta-p-equality} are proved in Section \ref{sec:bary}. Finally in Section \ref{sec:test-curve} we discuss the relation between $S^{(p)}$ and $d_p$-geometry and then prove a generalized version of Theorem \ref{thm:S-p=d_p/t}.

\textbf{Acknowledgments.}
The author would like to thank Tam\'as Darvas and Mingchen Xia for helpful discussions on Section \ref{sec:test-curve}. Special thanks go to Yuchen Liu for valuable comments and for proving Proposition \ref{prop:tilde-delta-p=delta-p}. He also thanks Yanir Rubinstein for suggesting Theorem \ref{thm:delta-p-equality}. The author is supported by the China post-doctoral grant BX20190014.

\section{Expected vanishing order with higher moments}
\label{sec:S-p}

Let $X$ be a klt projective variety and $L$ an ample line bundle on $X$. Also fix some $p\geq1$.

\subsection{Divisorial valuations}
The next definition is a natural generalization of the expected vanishing order introduced in \cite{FO18,BJ17}.
\begin{definition}
\label{def:S-p}
Let $p\in[1,+\infty)$.
Given any prime divisor $F$ over $X$, the $p$-th moment of the expected vanishing order of $L$ along $F$ at level $m$ is given by
$$
S_m^{(p)}(L,F):=\sup\bigg\{\frac{1}{d_m}\sum_{i=1}^{d_m}\bigg(\frac{\ord_F(s_i)}{m}\bigg)^p\bigg|\{s_i\}_{i=1}^{d_m}\text{ is a basis of }H^0(X,mL)\bigg\}.
$$
We also put
$$
S^{(p)}(L,F):=\lim_{m\rightarrow\infty}S^{(p)}_m(L,F),
$$
which is called $p$-th moment of the expected vanishing order of $L$ along $F$.
\end{definition}

This definition can be reformulated as follows (which in turn justifies the existence of the above limit).

\begin{lemma}
\label{lem:S-p-formula}
Given any prime divisor $F\subset Y\xrightarrow{\pi}X$, one has
$$
S_m^{(p)}(L,F)=\frac{1}{d_m}\sum_{j\geq1}\bigg(\frac{j}{m}\bigg)^p\cdot\bigg(h^0(m\pi^*L-jF)-h^0(m\pi^*L-(j+1)F)\bigg)
$$
and
$$
S^{(p)}(L,F)=\frac{1}{\vol(L)}\int_0^{\tau(L,F)}x^pd(-\vol(L-xF))=\frac{p}{\vol(L)}\int_0^{\tau(L,F)}x^{p-1}\vol(L-xF)dx.
$$
\end{lemma}

\begin{proof}
For the first statement, we follow the proof of \cite[Lemma 2.2]{FO18}. Given a basis $\{s_i\}$ of $H^0(X,mL)$, for integer $j\geq0$, let $a_j$ be the number of sections of $\{s_i\}$ that are contained in the subspace $H^0(m\pi^*L-jF)$ when pulled back to $Y$. Then one has
\begin{equation*}
    \begin{aligned}
    \frac{1}{d_m}\sum_{i=1}^{d_m}\bigg(\frac{\ord_F(s_i)}{m}\bigg)^p&=\frac{1}{d_m}\sum_{j\geq1}\bigg(\frac{j}{m}\bigg)^p\cdot(a_j-a_{j+1})\\
    &=\frac{1}{d_m}\sum_{j\geq 1}\bigg[\bigg(\frac{j}{m}\bigg)^p-\bigg(\frac{j-1}{m}\bigg)^p\bigg]\cdot a_j\\
    &\leq\frac{1}{d_m}\sum_{j\geq 1}\bigg[\bigg(\frac{j}{m}\bigg)^p-\bigg(\frac{j-1}{m}\bigg)^p\bigg]\cdot h^0(m\pi^*L-jF)\\
    &=\frac{1}{d_m}\sum_{j\geq1}\bigg(\frac{j}{m}\bigg)^p\cdot\bigg(h^0(m\pi^*L-jF)-h^0(m\pi^*L-(j+1)F)\bigg).
    \end{aligned}
\end{equation*}
The equality is achieved exactly when $\{s_i\}$ is compatible with the filtration $\{H^0(m\pi^*L-jF)\}_{j\geq0}$.

The second statement then follows from the theory of filtrated graded linear series and Newton--Okounkov bodies; see e.g. the proof of \cite[Lemma 2.7]{CRZ19} for an exposition.
\end{proof}

\begin{remark}
Very recently, using $\{S^{(p)}\}$, Han--Li \cite{HL20} constructed a non-Archimedean analogue of the $H$-functional of the K\"ahler--Ricci flow, which were previously studied by Tian--Zhang--Zhang--Zhu \cite{TZZZ13}, He \cite{He16} and Dervan--Sz\'ekelyhidi \cite{DS20}. More precisely, consider
$$
\sum_{k=1}^\infty \frac{(-1)^kS^{(k)}(L,F)}{k!}=\frac{1}{\vol(L)}\int_0^{\tau(L,F)}e^{-x}\vol(L-xF)dx. 
$$
Then Han--Li defined
$$
H^\mathrm{{NA}}(X,L):=\inf_F\bigg\{A_X(F)+\log\bigg(1-\frac{1}{\vol(L)}\int_0^{\tau(L,F)}e^{-x}\vol(L-xF)dx\bigg)\bigg\}.
$$
As shown by Han--Li \cite{HL20}, this invariant plays significant roles in the study of the Hamilton--Tian conjecture just as the $\delta$-invariant in the Yau--Tian--Donaldson conjecture.
\end{remark}

\subsection{Filtrations}
For simplicity we put
$$
R:=\bigoplus_{m\geq0}R_m
$$
with
$
R_m:=H^0(X,mL).
$
We say $\cF$ is a filtration of $R$ if for any $\lambda\in\RR$ and $m\in\NN$ there is a subspace $\cF^\lambda R_m\subset R_m$ satisfying
\begin{enumerate}
    \item $\cF^{\lambda^\prime}R_m\supseteq\cF^\lambda R_m$\ for any $\lambda^\prime\leq\lambda$;
    
    \item $\cF^\lambda R_m=\bigcap_{\lambda^\prime<\lambda}\cF^{\lambda^\prime}R_m$;
    
    \item 
    $\cF^{\lambda_1}R_{m_1}\cdot\cF^{\lambda_2}R_{m_2}\subseteq\cF^{\lambda_1+\lambda_2}R_{m_1+m_2}$ for any $\lambda_1,\lambda_2$ and $m_1,m_2\in\NN$;
    
    \item $\cF^\lambda R_m=R_m$ for $\lambda\leq0$ and $\cF^\lambda R_m=\{0\}$ for $\lambda\gg0$. 
    
\end{enumerate}
We say $\cF$ is \emph{linearly bounded} if
there exists $C>0$ such that $\cF^{Cm}R_m=\{0\}$ for any $m\in\NN$.
We say $\cF$ is a filtration of $R_m$ if only items $(1), (2)$ and $(4)$ are satisfied. We call $\cF$ \emph{trivial} if $\cF^\lambda R_m=\{0\}$ for any $\lambda>0$.

Following \cite{BJ17}, the definition of $S_m^{(p)}$ also extends to filtrations of $R_m$. More precisely, let $\cF$ be a filtration of $R_m$, the \emph{jumping numbers} of $\cF$ are given by
$$
0\leq a_{m,1}\leq a_{m,2}\leq\cdots\leq a_{m,d_m}
$$
where
\begin{equation}
    \label{eq:def-jump-number}
    a_{m,j}:=a_{m,k}(\cF):=\inf\{\lambda\in\RR_{\geq0}|\ \mathrm{codim}\cF^\lambda R_m\geq j\}.
\end{equation}
Then we put (see also \cite[(81)]{HL20})
$$
S^{(p)}_m(L,\cF):=\frac{1}{d_m}\sum_{j=1}^{d_m}\bigg(\frac{a_{m,j}}{m}\bigg)^p\text{ and }T_m(L,\cF):=\frac{a_{m,d_m}}{m}.
$$
Thus $S_m^{(1)}(L,\cF)=S_m(L,\cF)$ is the rescaled sum of jumping numbers studied in \cite{BJ17}.

A filtration is called an $\NN$-filtration if all its jumping numbers are non-negative integers. For instance the filtration induced by a divisor over $X$ is $\NN$-filtration. Given any filtration $\cF$ of $R_m$, its induced $\NN$-filtration $\cF_\NN$ is given by
\begin{equation}
    \label{eq:def-F-N}
    \cF^\lambda_\NN R_m:=\cF^{\lceil\lambda\rceil}R_m\ \text{for all }\lambda\in\RR_{\geq0}.
\end{equation}
Then one has
$$
a_{m,j}(\cF_{\NN})=\lfloor a_{m,j}(\cF)\rfloor.
$$
The next result is a simple generalization of \cite[Proposition 2.11]{BJ17}

\begin{proposition}
\label{prop:S=S-N}
If $\cF$ is a filtration on $R_m$, then
\begin{equation*}
    S_m^{(p)}(L,\cF)\geq S^{(p)}_m(L,\cF_\NN)\geq
    \begin{cases}
    S^{(p)}_m(L,\cF)-\frac{1}{m},\ p=1,\\
    S^{(p)}_m(L,\cF)-\frac{p}{m^{p-1}}S^{(1)}_m(L,\cF),\ 1<p<2,\\
    S^{(p)}_m(L,\cF)-\frac{p}{m}S^{(p-1)}_m(L,\cF),\ p\geq2.\\
    \end{cases}
\end{equation*}
\end{proposition}

\begin{proof}
This follows from the elementary inequality:
$$
(x-1)^p\geq x^p-px^{p-1}\ \text{for }p\geq1\text{ and }x\geq 1.
$$
\end{proof}

Now as in \cite[Section 2.2]{BJ17} let us fix a local system $(z_1,...,z_n)$ around a regular closed point of $X$, which then yields a Newton--Okounkov body 
$\Delta\subset\RR^n$ for $L$ and denote the Lebesgue measure on $\Delta$ by $\rho$. Then any filtration
$\cF$ of $R=R(X,L)=\oplus_{m\geq0}R_m$ induces a family of graded linear series $V_\bullet^t$ ($t\in\RR_{\geq0}$) and also a concave function $G$ on $\Delta$. More precisely, $V_\bullet^t=\oplus_{m}V^t_m$ with
$$
V_m^t:=\cF^{tm}R_m,
$$
which induces a Newton--Okounkov body $\Delta^t\subset\Delta$ and
$$
G(\alpha):=\sup\{t\in\RR_{\geq0}|\alpha\in\Delta^t\}.
$$
Then define
\begin{equation*}
    S^{(p)}(L,\cF):=\frac{1}{\vol(L)}\int_0^\infty pt^{p-1}\vol(V^t_{\bullet})dt=\frac{1}{\vol(L)}\int_0^\infty t^pd(-\vol(V^t_\bullet))=\frac{1}{\vol(\Delta)}\int_\Delta G^pd\rho
\end{equation*}
and also put
$$
T(L,\cF):=\lim_{m\rightarrow\infty}T_m(L,\cF).
$$
One then can deduce that
$$
\frac{\Gamma(p+1)\Gamma(n+1)}{\Gamma(p+n+1)}T(L,\cF)^p\leq S^{(p)}(L,\cF)\leq T(L,\cF)^p.
$$
This generalizes \cite[Lemma 2.6]{BJ17} (see also the proof of Theorem \ref{thm:p-barycenter-ineq}). A simple consequence is that
\begin{equation*}
\label{eq:T=lim-S-p}
    T(L,\cF)=\lim_{p\rightarrow\infty}S^{(p)}(L,\cF)^{1/p}.
\end{equation*}

The next result naturally generalizes Lemma 2.9, Corollary 2.10 and Proposition 2.11 in \cite{BJ17}. Since its proof is largely verbatim, we omit it.

\begin{proposition}
\label{prop:S-p-m->S-p}
The following statements hold.
\begin{enumerate}
    \item One has $S^{(p)}(L,\cF)=\lim_m S^{(p)}_m(L,\cF).$
    \item For every $\varepsilon>0$ there exists $m_0=m_0(\varepsilon)>0$ such that
    $$
    S^{(p)}_m(L,\cF)\leq(1+\varepsilon)S^{(p)}(L,\cF)
    $$
    for any $m\geq m_0$ and any linearly bounded filtration $\cF$ on $R$.
    \item If $\cF$ is a filtration on $R$, then
    $S^{(p)}(L,\cF_\NN)=S^{(p)}(L,\cF)$.
\end{enumerate}
\end{proposition}

Let $\Val_X$ be the set of real valuations on the function field of $X$ that are trivial on the ground field $\CC$.
Any $v\in\mathrm{Val}_X$ induces a filtration $\cF_v$ on $R$ via 
$$
\cF_v^{\lambda}R_m:=\{s\in R_m\ |\ v(s)\geq t\}
$$
for $m\in\NN$ and $\lambda\in\RR_{\geq0}$. Then put
$$
S^{(p)}_m(L,v):=S^{(p)}_m(L,\cF_v)
\text{ and }
S^{(p)}(L,v):=S^{(p)}(L,\cF_v).
$$


Note that $\cF_v$ is \emph{saturated} in the sense of \cite[Definition 4.4]{Fuj18-volume}. To be more precise, let
$$
\fb(|\cF_v^\lambda R_m|):=\mathrm{Im}\bigg(\cF^\lambda_v R_m\otimes(-mL)\rightarrow\cO_X\bigg)
$$
be the base ideal of $\cF^\lambda_v R_m$. Let $\overline{\fb(|\cF^\lambda_v R_m|)}$ denote its integral closure (i.e., $\overline{\fb(|\cF^\lambda_v R_m|)}$ is the set of elements $f\in\cO_X$ satisfying a monic equation $f^d+a_1f^{d-1}+...+a_d=0$ with $a_i\in\fb(|\cF_v^\lambda R_m|)^i$).

\begin{lemma}
For any $\lambda\in\RR_{\geq0}$ and $m\in\NN$, one has
$$
\cF^\lambda_v R_m=H^0(X,mL\otimes\fb(|\cF_v^\lambda R_m|))=H^0(X,mL\otimes\overline{\fb(|\cF^\lambda_v R_m|)}).
$$
\end{lemma}

\begin{proof}
Put
$
\fa_\lambda(v):=\{f\in\cO_X\ |\ v(f)\geq\lambda\}.
$
It is easy to see that $\fa_\lambda(v)$ is integrally closed, i.e., $\overline{\fa_\lambda(v)}=\fa_\lambda(v)$.
Moreover by definition,
$$\cF^\lambda_v R_m=H^0(X,mL\otimes\fa_\lambda(v)).$$
Thus
\begin{equation*}
    \begin{aligned}
        \cF^\lambda_v R_m\subset H^0(X,mL\otimes\fb(|\cF_v^\lambda R_m|))\subset H^0(X,mL\otimes\overline{\fb(|\cF^\lambda_v R_m|)})
        \subset H^0(X,mL\otimes\fa_\lambda(v))=\cF^\lambda_v R_m.\\
    \end{aligned}
\end{equation*}
So we conclude.
\end{proof}

For $t\in\RR_{\geq0}$ and $l\in\NN$, set
$$
V^t_{m,l}:=H^0(X,mlL\otimes\overline{\fb(|\cF^{tm}_v R_m|)^l}).
$$
The previous lemma implies that $V^t_{m,\bullet}:=\bigoplus_{l\in\NN}V^t_{m,l}$ is a subalgebra of $V^t_{\bullet}$. Put
$$
\Tilde{S}^{(p)}_m(L,v):=\frac{1}{m^n\vol(L)}\int_0^\infty pt^{p-1}\vol(V_{m,\bullet}^t)dt.
$$
As illustrated in \cite[Section 5]{BJ17}, the graded linear series  $V^t_{m,\bullet}$ can effectively approximate $V^t_{\bullet}$. The argument therein extends to our $L^p$ setting in a straightforward way. So we record the following result, which generalizes \cite[Theorem 5.3]{BJ17}, and leave its proof to the interested reader.

\begin{theorem}
\label{thm:S-p=tilde-S-p-m}
Let $X$ be a normal projective klt variety and $L$ an ample line bundle on $X$. Then there exists a constant $C=C(X,L)$ such that
$$
0\leq S^{(p)}(L,v)-\Tilde{S}^{(p)}_m(L,v)\leq\bigg(\frac{CA(v)}{m}\bigg)^p
$$
for all $m\in\NN^*$ and all $v\in\mathrm{Val}_X$ with $A_X(v)<\infty$
\end{theorem}

\section{Continuity in the big cone}
\label{sec:conti}

In this section we assume that $X$ is a klt projective variety and $L$ a big $\RR$-line bundle on $X$ (we refer to \cite{Laz-bookI} for the positivity notions of line bundles). As before, fix some $p\geq1$.
Recall that its $\delta^{(p)}$-invariant is given by
\begin{equation*}
    \delta^{(p)}(L):=\inf_F\frac{A_X(F)}{S^{(p)}(L,F)^{1/p}},
\end{equation*}
where $F$ runs through all the prime divisors over $X$.
The goal is to show Theorem \ref{thm:cont}. The proof is a slightly modified version of the one in \cite[Section 4]{Zhang20}. For the reader's convenience we give the details.

\begin{lemma}
\label{lem:delta-comparison}
There exists $\varepsilon_0$ only depending $n$ and $p$ such that the following holds. Given any big $\RR$-line bundle $L$ any $\varepsilon\in(0,\varepsilon_0)$, assume that there is a big $\RR$-line bundle $L_\varepsilon$ such that
$$
\text{both }(1+\varepsilon)L-L_\varepsilon\text{ and }L_\varepsilon-(1-\varepsilon)L\text{ are big,}
$$
we have
$$
\delta^{(p)}(L+\varepsilon L_\varepsilon)\leq\delta^{(p)}(L)\leq\delta^{(p)}(L-\varepsilon L_\varepsilon).
$$
\end{lemma}

\begin{proof}
We only prove $\delta^{(p)}(L+\varepsilon L_\varepsilon)\leq\delta^{(p)}(L)$, since the other part follows in a similar manner.
Let $F$ be any prime divisor over $X$. It suffices to show
$$
S^{(p)}(L+\varepsilon L_\varepsilon,F)\geq S^{(p)}(L,F).
$$
To this end, we calculate as follows:
\begin{equation*}
    \begin{split}
        S^{(p)}(L+\varepsilon L_\varepsilon,F)&=\frac{1}{\vol(L+\varepsilon L_\varepsilon)}\int_0^\infty px^{p-1}\vol(L+\varepsilon L_\varepsilon-xF)dx\\
        &\geq\frac{1}{\vol(L+(\varepsilon+\varepsilon^2)L)}\int_0^\infty px^{p-1}\vol(L+(\varepsilon-\varepsilon^2)L-xF)dx\\
        &=\bigg(\frac{1+\varepsilon-\varepsilon^2}{1+\varepsilon+\varepsilon^2}\bigg)^n\cdot S^{(p)}\bigg((1+\varepsilon-\varepsilon^2)L,F\bigg)\\
        &=\bigg(\frac{1+\varepsilon-\varepsilon^2}{1+\varepsilon+\varepsilon^2}\bigg)^n\cdot(1+\varepsilon-\varepsilon^2)^p\cdot S^{(p)}(L,F).\\
    \end{split}
\end{equation*}
By choosing $\varepsilon$ small enough we can arrange that
$$
\bigg(\frac{1+\varepsilon-\varepsilon^2}{1+\varepsilon+\varepsilon^2}\bigg)^n\cdot(1+\varepsilon-\varepsilon^2)^p\geq1.
$$
This completes the proof.
\end{proof}

\begin{proof}[Proof of Theorem \ref{thm:cont}]

Let $L$ be a big $\RR$-line bundle. Fix any auxiliary $\RR$-line bundle $S\in N^1(X)_\RR$. We need to show that, for any small $\varepsilon>0$, there exists $\gamma>0$ such that
$$
(1-\varepsilon)\delta^{(p)}(L)\leq\delta^{(p)}(L+\gamma S)\leq(1+\varepsilon)\delta^{(p)}(L).
$$
Here $L+\gamma S$ is always assumed to be big (by choosing $\gamma$ sufficiently small). Notice that for any $\varepsilon>0$, we can write
$$
L+\gamma S=\frac{1}{1+\varepsilon}\bigg(L+\varepsilon\big(L+\frac{(1+\varepsilon)\gamma}{\varepsilon}S\big)\bigg).
$$
Put
$$
L_\varepsilon:=L+\frac{(1+\varepsilon)\gamma}{\varepsilon}S.
$$
Then by choosing $\gamma$ small enough, we can assume that
$$
\text{both }(1+\varepsilon)L-L_\varepsilon\text{ and }L_\varepsilon-(1-\varepsilon)L\text{ are big.}
$$
So from the scaling property (easy to verify from the definition):
$$
\delta^{(p)}(\lambda L)=\lambda^{-1}\delta^{(p)}(L)\ \text{for }\lambda>0
$$
and Lemma \ref{lem:delta-comparison}, it follows that
$$
\delta^{(p)}(L+\gamma S)=(1+\varepsilon)\delta^{(p)}(L+\varepsilon L_\varepsilon)\leq(1+\varepsilon)\delta^{(p)}(L).
$$
We can also write
$$
L+\gamma S=\frac{1}{1-\varepsilon}\bigg(L-\varepsilon\big(L-\frac{(1-\varepsilon)\gamma}{\varepsilon}S\big)\bigg).
$$
Then a similar treatment as above yields
$$
\delta^{(p)}(L+\gamma S)\geq(1-\varepsilon)\delta^{(p)}(L).
$$
In conclusion, for any small $\varepsilon>0$, by choosing $\gamma$ to be sufficiently small, we have
$$
(1-\varepsilon)\delta^{(p)}(L)\leq\delta^{(p)}(L+\gamma S)\leq(1+\varepsilon)\delta^{(p)}(L).
$$
This completes the proof.
\end{proof}

\begin{remark}
Actually there is a more elegant proof of Theorem \ref{thm:cont} using the following argument provided by an anonymous referee: consider
$$
\bar{\delta}^{(p)}(L):=\frac{\delta^{(p)}(L)}{\vol(L)^{1/p}}=\inf_F\frac{A_X(F)}{\big(\int_0^{\tau(L,F)}px^{p-1}\vol(L-xF)dx\big)^{1/p}}.
$$
Observe that $\bar{\delta}^{(p)}(\cdot)$ satisfies:
\begin{enumerate}
    \item for any effective divisor $D$, one has $\bar{\delta}^{(p)}(L+D)\leq\bar{\delta}^{(p)}(L)$ since  $\vol(L+D-xF)\geq\vol(L-xF)$ for any $F$ over $X$ and any $x\geq0$; 
    
    \item  for any $\lambda\in\RR_{>0}$ one has
    $\bar{\delta}^{(p)}(\lambda L)=\lambda^{-\frac{n+p}{p}}\bar{\delta}^{(p)}(L)$.
\end{enumerate}
It is a well-known fact that these properties imply the continuity of $\bar{\delta}^{(p)}(\cdot)$ (the previous proof of Theorem \ref{thm:cont} is actually a variant of this fact). Now since both $\bar{\delta}^{(p)}(\cdot)$ and $\vol(\cdot)$ are continuous, so is $\delta^{(p)}(\cdot)$.
\end{remark}

\section{Semi-continuity in families}
\label{sec:lsc}
In this section we prove Theorem \ref{thm:delta-m-p-lsc}.
Let $(X,L)$ be a polarized klt pair.
We set
\begin{equation}
    \label{eq:def-tilde-delta-p}
    \Tilde{\delta}^{(p)}(L):=\inf_{v}\frac{A_X(v)}{S^{(p)}(L,v)^{1/p}}
\end{equation}
where $v$ runs through all the valuations with $A_X(v)<\infty$. 

By \cite[Theorem C]{BJ17}, $\Tilde{\delta}^{(p)}=\delta^{(p)}$ for $p=1$ and $\infty$, the main reason being that both $\alpha$ and $\delta$-invariants can be defined in terms of the log canonical threshold of certain divisors. However to show that $\Tilde{\delta}^{(p)}=\delta^{(p)}$ holds for general $p$ is more tricky; see Proposition \ref{prop:tilde-delta-p=delta-p}. Leaving this issue aside for the moment, we show that the semi-continuity established by Blum--Liu \cite[Theorem B]{BL18b} holds for $\Tilde{\delta}^{(p)}$ as well.

\begin{theorem}
\label{thm:delta-p-semi-continuous}
Let $\pi: X\rightarrow T$ is a projective family of varieties and $L$ is a $\pi$-ample Cartier divisor on $X$. Assume that $T$ is normal, $X_t$ is klt for all $t\in T$ and $K_{X/T}$ is $\QQ$-Cartier. Then the function
$$
T\ni t\mapsto\Tilde{\delta}^{(p)}(X_t,L_t)
$$
is lower semi-continuous.
\end{theorem}

In what follows we give a sketched proof.
To justify that the argument in \cite{BL18b} honestly extends to our setting, we need to spell out the main ingredients used in their proof. First of all, generalizing \cite[Proposition 4.10]{BL18b}, it is straightforward to obtain that
$$
\Tilde{\delta}^{(p)}(L)=\inf_\cF\frac{\lct(X,\fb_\bullet(\cF))}{S^{(p)}(L,\cF)^{1/p}},
$$
where $\cF$ runs through all non-trivial linearly bounded filtrations of $R$ and $\fb_\bullet(\cF)$ denote the graded ideal associated to $\cF$ (cf. \cite[Section 3.6]{BJ17}). Second, we need to introduce a quantized version of $\Tilde{\delta}^{(p)}$ by putting
$$
\Tilde{\delta}^{(p)}_m(L):=\inf_v\frac{A_X(v)}{S^{(p)}_m(L,v)^{1/p}},
$$
where $v$ runs through all the valuations with $A_X(v)<\infty$. Then by Proposition \ref{prop:S-p-m->S-p}, we have (as in \cite[Theorem 4.4]{BJ17})
\begin{equation}
    \label{eq:lim-tilde-delta-p-m}
    \Tilde{\delta}^{(p)}(L)=\lim_{m\rightarrow\infty}\Tilde{\delta}^{(p)}_m(L).
\end{equation}
Meanwhile, we also need
$$
\hat{\Tilde{\delta}}^{(p)}_m(L):=\inf_{\cF}\frac{\lct(X,\fb_\bullet(\hat{\cF}))}{S_m^{(p)}(L,\cF)^{1/p}}
$$
where $\cF$ runs through all non-trivial $\NN$-filtration of $R_m$ with $T_m(\cF)\leq1$ and $\hat{\cF}$ denote the filtration of $R$ generated by $\cF$ (cf. \cite[Definition 3.18]{BL18b}). Then combining Proposition \ref{prop:S=S-N} with the argument of \cite[Proposition 4.17]{BL18b}, we derive that
\begin{equation}
    \label{eq:delta-m-approxi}
    \bigg(\frac{1}{\Tilde{\delta}_m^{(p)}(L)}\bigg)^p-\frac{p}{m}\bigg(\frac{1}{\alpha(L)}\bigg)^p\leq\bigg(\frac{1}{\hat{\Tilde{\delta}}^{(p)}_m(L)}\bigg)^p\leq\bigg(\frac{1}{\Tilde{\delta}_m^{(p)}(L)}\bigg)^p.
\end{equation}

Now proceeding as in \cite{BL18b}, to conclude Theorem \ref{thm:delta-p-semi-continuous}, it  suffices to establish the following two results, which extend Theorem 5.2 and Proposition 6.4 in \cite{BL18b}.

\begin{theorem}
\label{thm:hat-delta-m<delta-m}
Let $\pi:X\rightarrow T$ be a projective $\QQ$-Gorenstein family of klt projective varieties over a normal base $T$ and $L$ a $\pi$-ample Cartier divisor on $X$. For each $\varepsilon>0$ there exits a positive integer $M=M(\varepsilon)$ such that
$$
\hat{\Tilde{\delta}}_m^{(p)}(X_t,L_t)-\Tilde{\delta}^{(p)}(X_t,L_t)\leq\varepsilon
$$
for all positive integer $m$ divisible by $M$ and $t\in T$.
\end{theorem}

\begin{proposition}
\label{prop:hat-delta-m-lsc}
Let $\pi:X\rightarrow T$ be a projective $\QQ$-Gorenstein family of klt projective varieties over a normal base $T$ and $L$ a $\pi$-ample Cartier divisor on $X$. For $m\gg0$, the function $T\ni t\mapsto \hat{\Tilde{\delta}}^{(p)}_m(X_t,L_t)$ is lower semi-continuous and takes finitely many values.
\end{proposition}

To show Theorem \ref{thm:hat-delta-m<delta-m}, an intermediate step is to prove (cf. also \cite[Proposition 5.16]{BL18b})
$$
\Tilde{\delta}^{(p)}_m(X_t,L_t)-\Tilde{\delta}^{(p)}(X_t,L_t)\leq\varepsilon.
$$
By the strategy of \cite{BL18b}, this can be derived from 
$$
S^{(p)}(L_t,v)-S_{m}^{(p)}(L_t,v)\leq \varepsilon A_{X_t}(v)^p,
$$
which can be proved by generalizing the argument of \cite[Theorem 5.13]{BL18b} to our $L^p$ setting 
(here we need to use Theorem \ref{thm:S-p=tilde-S-p-m}). Then using \eqref{eq:delta-m-approxi}, we conclude Theorem \ref{thm:hat-delta-m<delta-m}.

The proof of Proposition \ref{prop:hat-delta-m-lsc} is a verbatim generalization of \cite[Propositon 6.4]{BL18b} so we omit it. Thus by \cite[Proposition 6.1]{BL18b} we finish the proof of Theorem \ref{thm:delta-p-semi-continuous}.

Finally, to finish the proof of Theorem \ref{thm:delta-m-p-lsc}, it suffices to show the following result. The author is grateful to Yuchen Liu for providing the proof.

\begin{proposition}
\label{prop:tilde-delta-p=delta-p}
One has
$$
\delta^{(p)}(L)=\inf_{F\text{ \emph{over} }X}\frac{A_X(F)}{S^{(p)}(L,F)^{1/p}}=\inf_{v\in\mathrm{Val}_X}\frac{A_X(v)}{S^{(p)}(L,v)^{1/p}}.
$$
\end{proposition}

\begin{proof}
It amounts to proving
$$
\tilde{\delta}^{(p)}(L)\geq\delta^{(p)}(L).
$$
To show this, it is enough to show that this holds in the following quantized sense:
\begin{equation}
    \label{eq:tilde-delta-p-m>=delta-p-m}
    \hat{\tilde{\delta}}^{(p)}_m(L)\geq\inf_{F}\frac{A_X(F)}{S^{(p)}_m(L,F)^{1/p}}.
\end{equation}
Given this, then one can finish the proof by letting $m\rightarrow\infty$ as the left hand side converges to $\tilde{\delta}^{(p)}(L)$ by \eqref{eq:delta-m-approxi} and \eqref{eq:lim-tilde-delta-p-m} while the right hand converges to $\delta^{(p)}(L)$ by Proposition \ref{prop:S-p-m->S-p} and hence $\tilde{\delta}^{(p)}(L)\geq\delta^{(p)}(L)$ as desired.

To show \eqref{eq:tilde-delta-p-m>=delta-p-m}, the key point is that, given any $\NN$-filtration $\cF$ of $R_m$ with $T_m(\cF)\leq1$, the associated graded ideal $\fb_\bullet(\hat{\cF})$ is finitely generated (see \cite[Lemma 3.20 (2)]{BL18b}). Thus there is a prime divisor $F$ over $X$ computing $\lct(X,\fb_\bullet(\hat{\cF}))$. This then yields (as in \cite[Lemma 4.16]{BL18b})
$$
\frac{\lct(X,\fb_\bullet(\hat{\cF}))}{S_m^{(p)}(L,\cF)^{1/p}}\geq\frac{A_X(F)}{S^{(p)}(L,F)^{1/p}},
$$
which finishes the proof.
\end{proof}

\begin{proof}
[Proof of Theorem \ref{thm:delta-m-p-lsc}]
The result follows from Theorem \ref{thm:delta-p-semi-continuous} and Proposition \ref{prop:tilde-delta-p=delta-p}.
\end{proof}

\begin{remark}
Arguing as in \cite[Section 6]{BJ17}, one can also show that there is always a valuation $v^*$ computing $\Tilde{\delta}^{(p)}(L)$ when $L$ is ample. Moreover, as in \cite[Proposition 4.8]{BJ17}, $v^*$ is the unique valuation (up to scaling) computing $\lct(\frak{a}_\bullet(v^*))$. Then by \cite{Xu20}, $v^*$ is actually quasi-monomial (the author is grateful to one anonymous referee for providing this argument).
\end{remark}

\section{Barycenter inequalities}
\label{sec:bary}
Let $(X,L)$ be a polarized pair.
We prove Theorem \ref{thm:p-barycenter-ineq}, Theorem \ref{thm:delta-p=>KE} and Theorem \ref{thm:delta-p-equality} in this section. To show Theorem \ref{thm:p-barycenter-ineq}, the main ingredient is the Brunn--Minkowski inequality for Newton--Okoukov bodies (cf. \cite{LM09,ELMNP09}). While for Theorems \ref{thm:delta-p=>KE} and \ref{thm:delta-p-equality}, we also need some measure-theoretic argument from real analysis.

\begin{proof}[Proof of Theorem \ref{thm:p-barycenter-ineq}] We follow the argument in \cite{Fuj19b}.
To obtain the lower bound for $S^{(p)}(L,F)$, we use the fact that $\vol(L-xF)^{1/n}$ is a decreasing concave function (cf. \cite[Corollary 4.12]{LM09}), so that
$$
\vol(L-xF)\geq\frac{\vol(L)}{\tau^n(L,F)}\cdot\big(\tau(L,F)-x\big)^n
$$
and hence
\begin{equation*}
    \begin{aligned}
    S^{(p)}(L,F)&=\frac{p}{\vol(L)}\int_0^{\tau(L,F)}x^{p-1}\vol(L-xF)dx\\
    &\geq \frac{p}{\tau^n(L,F)}\int_0^{\tau(L,F)}x^{p-1}\big(\tau(L,F)-x\big)^ndx\\
    &=p\tau(L,F)^p\int_0^1t^{p-1}(1-t)^ndt\\
    &=\frac{\Gamma(p+1)\Gamma(n+1)}{\Gamma(p+n+1)}\tau^p(L,F).
    \end{aligned}
\end{equation*}

To show the upper bound for $S^{(p)}(L,F)$, we assume $n\geq2$ (the case of $n=1$ is trivial). We use the argument of \cite[Proposition 2.1]{Fuj19b}, which shows that there exists a non-negative concave function $f(x)$ for $x\in[0,\tau(L,F)]$ such that
$$
f^{n-1}(x)dx=d\bigg(\frac{-\vol(L-xF)}{\vol(L)}\bigg)\ \text{for }x\in(0,\tau(L,F)).
$$
Thus
\footnote{This suggests that $(S^{(p)}(L,F))^{1/p}$ can be treated as the $p$-th barycenter of a convex body along $x$-axis.}
$$
S^{(p)}(L,F)=\int_0^{\tau(L,F)}x^pf^{n-1}(x)dx.
$$
For simplicity set $b:=S^{(p)}(L,F)$. Note that $f(x)>0$ for $x\in(0,\tau(L,F))$. Then $b\in(0,\tau(L,F)^p)$ and by the concavity of $f(x)$, we have
\begin{equation*}
    \begin{cases}
    f(x)\geq\frac{f(b)}{b}x,\ x\in[0,b^{1/p}],\\
    f(x)\leq\frac{f(b)}{b}x,\ x\in[b^{1/p},\tau(L,F)].
    \end{cases}
\end{equation*}
Thus we have
$$
0=\int_0^{\tau(L,F)}(x^p-b)f^{n-1}(x)dx\geq\int_0^{\tau(L,F)}(x^p-b)\bigg(\frac{f(b)}{b}x\bigg)^{n-1}dx=\bigg(\frac{f(b)}{b}\bigg)^{n-1}\cdot\bigg(\frac{\tau(L,F)^{p+n}}{p+n}-b\frac{\tau^n(L,F)}{n}\bigg),
$$
which implies that
$$
b\leq\frac{n}{n+p}\tau(L,F)^p,
$$
as desired.
\end{proof}

An immediate consequence is the following, which generalizes \eqref{eq:alpha=delta}.
\begin{corollary}
One has 
$
\big(\frac{n+p}{n}\big)^{1/p}\alpha(L)\leq\delta^{(p)}(L)\leq\big(\frac{\Gamma(p+n+1)}{\Gamma(p+1)\Gamma(n+1)}\big)^{1/p}\alpha(L).
$
\end{corollary}

Now we turn to the proof of Theorem \ref{thm:delta-p=>KE}. The key point is the following monotonicity.

\begin{proposition}
\label{prop:H-p-increase}
Let $F$ be any prime divisor over $X$. Set
$$
H(p):=\bigg(\frac{n+p}{n}S^{(p)}(L,F)\bigg)^{1/p} \text{ for }p\geq1.
$$
Then $H(p)$ is non-decreasing in $p$.
\end{proposition}

\begin{proof}
When $n=1$, $H(p)$ is a constant, so there is nothing to prove.
Then as in the previous proof, we assume $n\geq2$ and write
$$
S^{(p)}(L,F)=\int_0^{\tau(L,F)}x^pf^{n-1}(x)dx
$$
for some non-negative concave function $f(x)$ defined on $[0,\tau(L,F)]$. For simplicity set $\tau:=\tau(L,F)$. Then it amounts to proving that
$$
H(p)=\bigg(\frac{n+p}{n}\int_0^\tau x^pf^{n-1}(x)dx\bigg)^{1/p}
$$
is non-decreasing in $p$ for any non-negative concave function $f(x)$ defined on $[0,\tau]$ with the normalization condition
$$
\int_0^\tau f^{n-1}(x)dx=1.
$$

To this end, we introduce an auxiliary function:
$$
g(x):=\frac{f(x)}{x},\ x\in(0,\tau].
$$
By concavity of $f(x)$,
$
g(x) \text{ is differentiable almost everywhere, non-negative and decreasing for }x\in (0,\tau].
$
For $s>n-1$, we put
$$
K(s):=s\int_0^\tau x^{s-1}g^{n-1}(x)dx.
$$
Then one has (using $\int_0^\tau f^{n-1}(x)dx=1$)
$$
H(p)=\bigg(\frac{K(n+p)}{K(n)}\bigg)^{1/p}.
$$
To show this is non-decreasing in $p$, it then suffices to show that
$$
K(s)\text{ is log convex for } s>n-1. 
$$

To see this, for each small $\varepsilon>0$,
using integration by parts for Lebesgue–Stieltjes integration, we have
$$
s\int_\varepsilon^\tau x^{s-1}g^{n-1}(x)dx=\int_\varepsilon^\tau g^{n-1}(x)dx^s=\int_\varepsilon^\tau x^sd(-g^{n-1}(x))+\tau^sg^{n-1}(\tau)-\varepsilon^{s}g^{n-1}(\varepsilon).
$$
Here we used the fact that, as a monotonic function, $g^{n-1}$ has bounded variation on $[\varepsilon,\tau]$ and hence the measure $d(-g^{n-1}(x))$ is well-defined on $(\varepsilon,\tau)$. Now observing
$$
\lim_{\varepsilon\searrow 0}\varepsilon^sg^{n-1}(\varepsilon)=\lim_{\varepsilon\searrow 0}\varepsilon^{s+1-n}f(\varepsilon)=0\text{ for }s>n-1,
$$
we derive that for $s>n-1$,
$$
K(s)=\int_0^\tau x^s d(-g^{n-1}(x))+\tau^sg^{n-1}(\tau),
$$
where $d(-g^{n-1}(x))$ is understood as a measure on $(0,\tau)$. Set
$$
M(s):=\int_0^\tau x^s d(-g^{n-1}(x))\text{ for }s>n-1.
$$
Now applying H\"older's inequality to the measure space $\big((0,\tau),d(-g^{n-1}(x))\big)$, we derive that
\begin{equation}
    \label{eq:Holder}
    \bigg(\int_0^\tau\log x \cdot x^sd(-g^{n-1}(x))\bigg)^2\leq\bigg(\int_0^\tau(\log x)^2x^sd(-g^{n-1}(x))\bigg)\bigg(\int_0^\tau x^sd(-g^{n-1}(x))\bigg),
\end{equation}
which reads
$$
M^{\prime\prime}M-(M^\prime)^2\geq0.
$$
This implies that
$$
K^{\prime\prime}K-(K^\prime)^2=M^{\prime\prime}M-(M^\prime)^2+\tau^sg^{n-1}(\tau)\int_0^\tau\big(\log\frac{x}{\tau}\big)^2x^sd(-g^{n-1}(x))\geq0,
$$
so that $K(s)$ is log convex, as desired.
\end{proof}

\begin{remark}
We also believe that 
$$
\bigg(\frac{\Gamma(n+p+1)}{\Gamma(n+1)\Gamma(p+1)}S^{(p)}(L,F)\bigg)^{1/p} \text{ is non-increasing in }p.
$$
However it seems to the author that the proof of this is much more difficult, which may involve the log concavity of the generalized beta function. If this is indeed true, it then follows that
$$
\delta^{(p)}(L)\geq\frac{1}{n+1}\bigg(\frac{\Gamma(n+p+1)}{\Gamma(n+1)\Gamma(p+1)}\bigg)^{1/p}\delta(L),
$$
which hence generalizes the inequality $\alpha(L)\geq\frac{1}{n+1}\delta(L)$. Meanwhile it will also follow that $\delta^{(p)}(L)$ is continuous in $p$.
We leave this problem to the interested readers.
\end{remark}

\begin{proof}
[Proof of Theorem \ref{thm:delta-p=>KE}]
Proposition \ref{prop:H-p-increase} implies that for any $F$ over $X$ and $p\geq1$,
\begin{equation}
    \label{eq:S-p>S}
    S^{(p)}(L,F)^{1/p}\geq\frac{n+1}{n}\bigg(\frac{n}{n+p}\bigg)^{1/p}S(L,F).
\end{equation}
Thus
$$
\delta^{(p)}(L)=\inf_F\frac{A_X(F)}{S^{(p)}(L,F)^{1/p}}\leq\frac{n}{n+1}\bigg(\frac{n+p}{n}\bigg)^{1/p}\inf_F\frac{A_X(F)}{S(L,F)}=\frac{n}{n+1}\bigg(\frac{n+p}{n}\bigg)^{1/p}\delta(L).
$$
Now in the Fano setting (when $L=-K_X$), we finish the proof by invoking Theorem \ref{theorem:delta} and \cite{BBJ18,LTW19}.
\end{proof}

Finally, we prove Theorem \ref{thm:delta-p-equality}. This boils down to a carefully analysis on the equality case in the proof Proposition \ref{prop:H-p-increase}.

\begin{proof}
[Proof of Theorem \ref{thm:delta-p-equality}] The case of $p=\infty$ is exactly \cite[Theorem 1.2]{Fuj19c}. So we assume $p\in(1,\infty)$. We follow the strategy of Fujita \cite{Fuj19c}. Note that $X=\PP^1$ clearly satisfies the claimed statement, so assume that $n\geq2$ and that $X$ is not K-stable with $\delta(-K_X)=1$. Then by \cite[Theorem 1.6]{Fuj19} there exists a dreamy divisor $F$ over $X$ such that
$$
A_X(F)=S(-K_X,F).
$$
Moreover this $F$ has to achieve the equality in \eqref{eq:S-p>S}, namely
$$
S^{(p)}(L,F)^{1/p}=\frac{n+1}{n}\bigg(\frac{n}{n+p}\bigg)^{1/p}S(L,F).
$$
Using the notion in the proof of Proposition \ref{prop:H-p-increase}, this reads
$$
\frac{n+1}{n}\int_0^\tau xf^{n-1}(x)dx=\bigg(\frac{n+p}{n}\int_0^\tau x^pf^{n-1}(x)dx\bigg)^{1/p}.
$$
Namely,
$$
\frac{K(n+1)}{K(n)}=\bigg(\frac{K(n+p)}{K(n)}\bigg)^{1/p}.
$$
By the log convexity of $K(s)$, this forces that
$$\log K(s)\text{ is affine for }s\in[n,n+p].$$
In particular
$$
K^{\prime\prime}K-(K^\prime)^2=M^{\prime\prime}M-(M^\prime)^2+\tau^sg^{n-1}(\tau)\int_0^\tau\big(\log\frac{x}{\tau}\big)^2x^sd(-g^{n-1}(x))=0\text{ for }s\in[n,n+p].
$$
This further forces that
$$
M^{\prime\prime}M-(M^\prime)^2=\tau^sg^{n-1}(\tau)\int_0^\tau\big(\log\frac{x}{\tau}\big)^2x^sd(-g^{n-1}(x))=0 \text{ for }s\in[n,n+p].
$$
In particular, the H\"older inequality \eqref{eq:Holder} is an equality, which implies that there exist real numbers $\alpha, \beta\geq0$, not both of them zero, such that
$$
\alpha|\log x|=\beta x^{s/2} \text{ holds  for }x\in(0,\tau)\backslash U,
$$
where $U\subset(0,\tau)$ is a subset satisfying
$$
\int_Ud(-g^{n-1}(x))=0.
$$
This implies that $d(-g^{n-1}(x))$ is a zero measure away from finitely many points in $(0,\tau)$. Therefore
$$
g(x)\text{ is a step function.}
$$
Now recall that $f(x)=xg(x)$ is concave and hence continuous
on $(0,\tau)$. Thus
$$
f(x)=Cx \text{ for some constant }C>0.
$$
The normalization condition $\int_0^\tau f^{n-1}(x)dx=1$ further implies that
$$
C^{n-1}=\frac{n}{\tau^n}.
$$
Thus
$$
A_X(F)=S(-K_X,F)=\frac{n}{\tau^n}\int_0^\tau x^ndx=\frac{n}{n+1}\tau(-K_X,F).
$$
Then by \cite[Theorem 4.1]{Fuj19c}, $X\cong\PP^n$. Now let $H$ be a hyperplane in $\PP^n$, straightforward calculation then yields
$$
\frac{n}{n+1}\bigg(\frac{n+p}{n}\bigg)^{1/p}=\delta^{(p)}(-K_X)\leq\frac{A_X(H)}{S^{(p)}(-K_X,H)^{1/p}}=\frac{1}{n+1}\cdot\bigg(\frac{\Gamma(p+n+1)}{\Gamma(p+1)\Gamma(n+1)}\bigg)^{1/p}.
$$
This will give us a contradiction. Indeed, consider the function
$$
h(x):=x\log n-\sum_{i=1}^{n-1}\log(\frac{x+i}{i})\text{ for }x\geq1.
$$
Observe that
$$
h(1)=\log n-\log n=0.
$$
Moreover, for $x\geq1$,
\begin{equation*}
    \begin{aligned}
        h^\prime(x)&=\log n-\sum_{i=1}^{n-1}\frac{1}{x+i}
        \geq\log n-\sum_{i=2}^{n}\frac{1}{i}>\log n-\int_1^n\frac{1}{x}dx=0.\\
    \end{aligned}
\end{equation*}
Here we used $n\geq2$. Thus
$$
h(p)=p\log n-\sum_{i=1}^{n-1}\log\frac{p+i}{i}>0
$$
since we assumed $p>1$.
From this we derive that (recall $x\Gamma(x)=\Gamma(x+1)$)
\begin{equation*}
    \begin{aligned}
        \frac{1}{n+1}\cdot\bigg(\frac{\Gamma(p+n+1)}{\Gamma(p+1)\Gamma(n+1)}\bigg)^{1/p}&=\frac{1}{n+1}\cdot\bigg(\prod_{i=1}^n\frac{p+i}{i}\bigg)^{1/p}
        <\frac{n}{n+1}\bigg(\frac{n+p}{n}\bigg)^{1/p},\\
    \end{aligned}
\end{equation*}
which is a contradiction.
So we conclude.
\end{proof}

\section{Relating to the $d_p$-geometry of maximal geodesic rays}
\label{sec:test-curve}

We prove Theorem \ref{thm:S-p=d_p/t} in this section. In fact we will carry out the discussion in a more general fashion using filtrations instead of divisorial valuations. 

Our setup is as follows.
Let $(X,L)$ be a polarized K\"ahler manifold. We fix a smooth Hermitian metric $h$ on $L$ such that $\omega:=-dd^c\log h\in c_1(L)$ defines a K\"ahler form (here $dd^c:=\frac{\sqrt{-1}}{2\pi}\partial\bar{\partial}$).
As before put
$
R:=\bigoplus_{m\geq0}R_m
$
with
$
R_m:=H^0(X,mL).
$
Let $\cF$ be a linearly bounded filtration of $R$. Then $\cF$ will induce a \emph{test curve}, or equivalently (by Legendre transform), a \emph{maximal geodesic ray} in the space of pluri-subharmonic potentials (see \cite{RWN14,DX20} for a detailed discussion). These objects play crucial roles in the study of the Yau--Tian--Donaldson conjecture (cf. e.g. \cite{BBJ18,Li20}).

\begin{definition}\cite[Section 7]{RWN14}
For any $\lambda\in\RR$ and $x\in X$, put
$$
\psi^{\cF}_{\lambda,m}(x):=\sup\bigg\{\frac{1}{m}\log|s|^2_{h^m}(x)\bigg|s\in\cF^{\lambda m}R_m,\ \sup_X|s|^2_{h^m}\leq1\bigg\}
$$
and
$$
\psi^{\cF}_\lambda:=\bigg(\lim_{m\rightarrow+\infty}\psi_{\lambda,m}\bigg)^*,
$$
where $*$ denotes the upper semi-continuous regularization. We call $\psi_\lambda^{\cF}$ the \emph{test curve} induced by $\cF$.
\end{definition}

Note that $\psi_\lambda^{\cF}$ is non-increasing and concave in $\lambda$ (since $\cF$ is decreasing and multiplicative).

\begin{theorem}
\cite[Corollary 7.12]{RWN14}
\label{thm:varphi-t-sup-psi-tau}
Consider the Legendre transform
\begin{equation}
\label{eq:def-Legendre-transf}
    \varphi^{\cF}_t:=\bigg(\sup_{\lambda\in\RR}\{\psi^{\cF}_\lambda+t\lambda\}\bigg)^*\text{ for }t\geq0.
\end{equation}
Then $\varphi^{\cF}_t$ is a weak geodesic ray emanating from $0$.
\end{theorem}

By \emph{weak geodesic} we mean that $\varphi_t^\cF$ satisfies certain homogeneous Monge--Amp\`ere equation in a weak sense (cf. \cite{RWN14}). 
Note that $\varphi_t^\cF$ is in fact \emph{maximal} in the sense of \cite[Definition 6.5]{BBJ18} (see \cite[Example 6.9]{BBJ18}), so we also call $\varphi^{\cF}_t$ the \emph{maximal geodesic ray induced by $\cF$}. Recently it is shown by Darvas--Xia \cite[Proposition 3.6]{DX20} that the upper semi-continuous regularization in \eqref{eq:def-Legendre-transf} is unnecessary. A priori, the regularity of $\varphi^{\cF}_t$ could be rather weak. But when $\cF$ is a filtration induced by an \emph{test configuration} (in the sense of \cite{WN12}), $\varphi^\cF_t$ has $C^{1,1}$ regularity in $t$ and $x$ variables by \cite[Theorem 9.2]{RWN14} and \cite[Theorem 1.2]{CTW18} (see also \cite[Theorem 1.3]{PS10}).

\begin{remark}
When $\cF=\cF_v$ for some $v\in\mathrm{Val}_X$, we put
$\varphi_t^v:=\varphi_t^{\cF_v}$. When $v=\ord_F$ for some prime divisor $F$ over $X$, we also write $\varphi^{F}_t:=\varphi^{\ord_F}_t$. This explains the notation in Theorem \ref{thm:S-p=d_p/t}.
\end{remark}

Note that $\varphi^{\cF}_t(x)$ is convex in $t$.
Dually, one further has
\begin{equation}
\label{eq:psi=inf-varphi}
    \psi^{\cF}_\lambda=\inf_{t\geq0}\{\varphi^{\cF}_t-t\lambda\}.
\end{equation}
See \cite{RWN14} for the proof.

An equivalent way of producing the geodesic ray $\varphi^{\cF}_t$ is by quantization approach. More precisely, for $m\geq1$, let $\{a_{m,i}\}_{1\leq i\leq d_m}$ be the set of jumping numbers of $\cF$ (recall \eqref{eq:def-jump-number}). Now consider the Hermitian inner product 
$$
H_m:=\int_Xh^m(\cdot,\cdot)\omega^n
$$
on $R_m$. By elementary linear algebra one can find an  $H_m$-orthonormal basis $\{s_i\}_{1\leq i\leq d_m}$ of $R_m$ such that
$$
s_i\in\cF^{a_{m,i}}R_m\text{ for each }1\leq i\leq d_m.
$$
Now set
\begin{equation*}
    \label{eq:def-varphi-F-t-m}
    \varphi^{\cF}_{t,m}:=\frac{1}{m}\log\sum_{i=1}^{d_m} e^{a_{m,i}t}|s_i|^2_{h^m},\ t\geq0.
\end{equation*}
One can easily verify that $\varphi^{\cF}_{t,m}$ does not depend on the choice of $\{s_i\}$.
We call $\varphi^{\cF}_{t,m}$ the \emph{Bergman geodesic ray induced by $\cF$}.
Such geodesic ray goes back to the work of Phong--Sturm \cite{PS07} and is used to construct geodesic rays in the space of K\"ahler potentials by approximation. The above Bergman geodesic ray has also been utilized in the recent work \cite{RTZ20} to study quantized $\delta$-invariants. 

\begin{theorem}
\cite[Theorem 9.2]{RWN14}
One has
$$
\varphi^{\cF}_t=\lim_{m\rightarrow+\infty}\bigg[\sup_{k\geq m}\varphi^{\cF}_{t,k}\bigg]^*.
$$
\end{theorem}

We remark that although \cite[Theorem 9.2]{RWN14} is only stated for filtrations that are induced from test configurations, one can easily verify that the argument therein works for general filtrations. 

The next standard result shows that $\varphi^\cF_t$ has linear growth.
\begin{lemma}
\label{lem:varphi-F-t-linearly-growth}
One has
$$
0\leq\varphi^\cF_t\leq T(L,\cF)t\text{ for any }t\geq0.
$$
\end{lemma}

\begin{proof}
We clearly have
$$
\frac{1}{m}\log\sum_{i=1}^{d_m}|s_i|^2_{h^m}\leq\varphi^{\cF}_{t,m}\leq\frac{1}{m}\log\sum_{i=1}^{d_m}|s_i|^2_{h^m}+T_m(L,\cF)t.
$$
Meanwhile, the standard Tian--Yau--Zelditch expansion implies that
$$
\lim_{m\rightarrow\infty}\frac{1}{m}\log\sum_{i=1}^{d_m}|s_i|^2_{h^m}=0.
$$
So the assertion follows from the previous theorem. 
\end{proof}

\begin{lemma}
\label{lem:varphi-F=varphi-F-NN}
One has
$$
\varphi^{\cF_\NN}_t=\varphi^\cF_t.
$$
\end{lemma}

\begin{proof}
Recall \eqref{eq:def-F-N} that $\cF_\NN$ is the $\NN$-filtration induced by $\cF$. The jumping numbers of $\cF_\NN$ are given by $\{\lfloor a_{m,i}\rfloor\}_{1\leq i\leq d_m}$. Let as above $\{s_i\}$ be an $H_m$-orthonormal basis such that $s_i\in\cF^{a_{m,i}}R_m$ for any $1\leq i\leq d_m$. Then observe that
$$
s_i\in\cF^{a_{m,i}}R_m\subset\cF^{\lfloor a_{m,i}\rfloor}R_m=\cF_\NN^{\lfloor a_{m,i}\rfloor}R_m\text{ for any }1\leq i\leq d_m.
$$
So $\{s_i\}$ is also compatible with $\cF_\NN$ and we have
$$
\varphi^{\cF_\NN}_{t,m}=\frac{1}{m}\log\sum_{i=1}^{d_m}e^{\lfloor a_{m,i}\rfloor t}|s_i|^2_{h^m}.
$$
Then clearly
$$
\varphi^\cF_{t,m}-\frac{t}{m}\leq\varphi^{\cF_\NN}_{t,m}\leq\varphi^\cF_{t,m}.
$$
So the assertion follows from the previous theorem.
\end{proof}

The main contribution of this section is the following result, which includes Theorem \ref{thm:S-p=d_p/t} as a special case. 

\begin{theorem}
\label{thm:S-p=d_p/t-filtration}
For any linearly bounded filtration $\cF$ one has
$$
S^{(p)}(L,\cF)^{1/p}=\frac{d_p(0,\varphi^{\cF}_t)}{t}\text{ for all }t>0,
$$
where $d_p$ denotes the Finsler metric introduced by Darvas (see \eqref{eq:def-d-p}).
\end{theorem}

The proof will be divided into several steps. Firstly, by Proposition \ref{prop:S-p-m->S-p} and Lemma \ref{lem:varphi-F=varphi-F-NN}, we may assume without loss of generality that $\cF$ is an $\NN$-filtration. Then we will prove Theorem \ref{thm:S-p=d_p/t-filtration} under the additional assumption that
$$
\varphi^\cF_t\text{ has }C^1\text{ regularity in both }x\text{ and }t\text{ variables.}
$$
As we have mentioned, this regularity assumption automatically holds when $\cF$ is induced by a test configuration, in which case $\cF$ is a \emph{finitely generated} $\NN$-filtrations (see \cite[Section 2.5]{BHJ17}). 
We will eventually drop this assumption as any linearly bounded $\NN$-filtration can be approximated by a sequence of finitely generated $\NN$-filtrations in a suitable sense.

Note that the time derivative $\dot\varphi^\cF_t(x)$ is well-defined for any point $x\in X$ and $t\geq0$ when $\varphi^\cF_t$ is $C^1$.

\begin{lemma}
Assume that $\varphi^\cF_t$ has $C^1$ regularity.
Then for each point $x\in X$ and $t\geq0$, $\dot\varphi^\cF_t(x)\geq0$.
\end{lemma}

\begin{proof}
By convexity $\dot\varphi^\cF_t(x)$ is non-decreasing in $t$, so it suffices to show that
$$
\dot\varphi^\cF_0(x)\geq0.
$$
This follows readily from the construction (see also \cite[Lemma 4.2]{His16}).
Indeed, by $\cF^0R_m=R_m$ and the Tian--Yau--Zelditch expansion, one clearly has
$$\psi^\cF_0(x)=\varphi^\cF_0(x)=0.$$
Thus by \eqref{eq:psi=inf-varphi},
$$
\dot\varphi^\cF_0(x)=\lim_{t\rightarrow 0^+}\frac{\varphi^\cF_t(x)}{t}\geq\lim_{t\rightarrow 0^+}\frac{\psi_0^\cF(x)}{t}=0,
$$
as desired.
\end{proof}

The non-negativity of $\dot\varphi^\cF_t$ makes it convenient to consider the following $L^p$ Finsler speed of the geodesic; see the survey of Darvas \cite{DarvasSurvey} for more information on this subject. When $\varphi^\cF_t$ is $C^1$, set for $p\geq1$
\begin{equation}
\label{eq:L-p-speed}
    ||\dot\varphi^\cF_t||_p:=\bigg(\frac{1}{V}\int_X(\dot\varphi^\cF_t)^p(\omega+dd^c\varphi^\cF_t)^n\bigg)^{1/p}.
\end{equation}
Here $V:=\vol(L)=\int_X\omega^n$ and $(\omega+dd^c\varphi^\cF_t)^n$ is understood as the \emph{non-pluripolar} Monge--Amp\`ere measure defined in \cite{BEGZ10}. Since $\varphi^\cF_t$ has linear growth, $(\omega+dd^c\varphi^\cF_t)^n$ also coincides with the classical definition of Bedford--Taylor \cite{BT76}. As shown in \cite{DarvasSurvey}, \eqref{eq:L-p-speed} is independent of $t$. The $d_p$-distance from $0$ to $\varphi^
{\cF}_t$ is then given by
\begin{equation}
    \label{eq:def-d-p}
    d_p(0,\varphi^\cF_t):=\int_0^t||\dot\varphi^{\cF}_s||_pds.
\end{equation}
When $\varphi^\cF_t$ lacks $C^1$ regularity, one can still make sense of $d_p(0,\varphi^\cF_t)$ by using $C^{1,1}$ decreasing sequence that converges to $\varphi^\cF_t$; see \cite{DarvasSurvey} for detail.

\begin{proposition}
\label{prop:S-p=p-norm}
Assume that $\varphi^\cF_t$ has $C^1$ regularity then one has
$$
S^{(p)}(L,\cF)^{1/p}=||\dot\varphi^\cF_t||_p \text{ for any $t\geq0$}.
$$
\end{proposition}


\begin{proof}
We recall the main result of Hisamoto \cite{His16}, who dealt with filtrations that are induced from test configurations. But note that his argument also works for general filtrations that give rise to $C^1$ geodesic rays, from which we deduce that
for each $t\geq0$
$$
\frac{n!}{m^n}\sum_{i=1}^{d_m}\delta_{\frac{a_{m,i}}{m}}\xrightarrow{m\rightarrow+\infty}(\dot\varphi^\cF_t)_*(\omega+dd^c\varphi^\cF_t)^n 
$$
as measures on $\RR$; here recall that $\{a_{m,i}\}$ is the set of jumping numbers of $\cF$.
Thus we obtain that
\begin{equation*}
    \begin{aligned}
        S^{(p)}(L,\cF)&=\lim_{m\rightarrow+\infty}S_m^{(p)}(L,\cF)\\
        &=\lim_{m\rightarrow+\infty}\frac{1}{d_m}\sum_{i=1}^{d_m}\bigg(\frac{a_{m,i}}{m}\bigg)^p\\
        &=\frac{1}{\vol(L)}\int_\RR x^p(\dot\varphi^\cF_t)_*(\omega+dd^c\varphi^\cF_t)^n\\
        &=\frac{1}{V}\int_X(\dot\varphi_t^\cF)^p(\omega+dd^c\varphi^\cF_t)^n.
    \end{aligned}
\end{equation*}
This completes the proof.
\end{proof}

Thus we have shown Theorem \ref{thm:S-p=d_p/t-filtration} under the assumption that $\varphi^\cF_t$ has $C^1$ regularity. Now we show how to drop this assumption by a standard approximation procedure.

\begin{lemma}
\label{lem:approx-S-p-by-F-m}
Let $\cF$ be a linearly bounded filtration of $R$. Let
$\{\cF_m\}_{m\in\NN_{>0}}$ be a sequence of filtrations of $R$ that satisfies
\begin{enumerate}
    \item $\cF^{\lambda}_mR_m=\cF^\lambda R_m$ for any $m\geq1$ and $\lambda\in\RR_{\geq0}$.
    \item $\cF^{\lambda}_mR_k\subset\cF^\lambda R_k$ for any $k\geq m$ and $\lambda\in\RR_{\geq0}$.
\end{enumerate}
Then it holds that
$$
S^{(p)}(L,\cF)=\lim_{m\rightarrow\infty}S^{(p)}(L,\cF_m).
$$
\end{lemma}

\begin{proof}
Observe that each $\cF_m$ is linearly bounded by our assumption. By Proposition \ref{prop:S-p-m->S-p}(2), there exists $\varepsilon_m\rightarrow 0$ as $m\rightarrow\infty$ such that
\begin{equation*}
    \begin{aligned}
        S^{(p)}(L,\cF_m)&\geq(1-\varepsilon_m)S^{(p)}_m(L,\cF_m)=(1-\varepsilon_m)S^{(p)}_m(L,\cF).\\
    \end{aligned}
\end{equation*}
Sending $m\rightarrow\infty$ we find that
$$
\liminf_{m\rightarrow\infty} S^{(p)}(L,\cF_m)\geq S^{(p)}(L,\cF).
$$
On the other hand, Proposition \ref{prop:S-p-m->S-p}(1) yields
$$
S^{(p)}(L,\cF_m)=\lim_{k\rightarrow\infty}S_k^{(p)}(L,\cF_m)\leq\lim_{k\rightarrow\infty}S_k^{(p)}(L,\cF)=S^{(p)}(L,\cF),
$$
so that
$$
\limsup_{m\rightarrow\infty} S^{(p)}(L,\cF_m)\leq S^{(p)}(L,\cF).
$$
This completes the proof. 
\end{proof}

Now given any linearly bounded $\NN$-filtration $\cF$ of $R$, one can construct a sequence of finitely generated $\NN$-filtrations $\{\cF_m\}$ that ``converges'' to $\cF$ as follows (see also \cite{Sz15,BL18b}). For each $m\geq1$, $\cF_m$ is given by:

\begin{enumerate}
    \item For $k<m$,
    \begin{equation*}
    \cF^p_m R_k:=
        \begin{cases}
        R_k\ &\text{ for }p=0,\\
        0\ &\text{ for }p>0.\\
        \end{cases}
    \end{equation*}
    \item For $k=m$,
    $$
    \cF^p_m R_m:=\cF^p R_m\text{ for }p\in\NN.
    $$
    \item For $k>m$ and $p\in\NN$,
    $$
    \cF^p_m R_k:=\sum_{b}\bigg((\cF^1 R_m)^{b_1}\cdot\cdots\cdot(\cF^{mT_m(L,\cF)}R_m)^{b_{mT_m(L,\cF)}}\bigg)\cdot R_{k-m\sum_{i=1}^{mT_m(L,\cF)} b_i},
    $$
    where the sum runs through all $b=(b_1,...,b_{mT_m(L,\cF)})\in\NN^{mT_m(L,\cF)}$ such that
    $\sum_{i=1}^{mT_m(L,\cF)}ib_i\geq p$ and $k\geq m\sum_{i=1}^{mT_m(L,\cF)}b_i$.
\end{enumerate}
Then $\{\cF_m\}$ is a sequence of finitely generated $\NN$-filtrations of $R$ that satisfies the conditions in Lemma \ref{lem:approx-S-p-by-F-m}. So we have
$$
S^{(p)}(L,\cF)=\lim_{m\rightarrow\infty}S^{(p)}(L,\cF_m).
$$
Moreover, $\cF_m$ is induced by test configurations (see \cite[Section 2.5]{BHJ17}) since each $\cF_m$ is finitely generated, and hence the associated maximal geodesic ray $\varphi^{\cF_m}_t$ has $C^1$
regularity (see \cite{PS10,CTW18}). So Proposition \ref{prop:S-p=p-norm} yields
$$
S^{(p)}(L,\cF_m)^{1/p}=\frac{d_p(0,\varphi^{\cF_m}_t)}{t}\text{ for any }t> 0.
$$
Then to finish the proof of Theorem \ref{thm:S-p=d_p/t-filtration}, it remains to show the following

\begin{proposition}
\label{prop:d-p-phi-F-t-approxi}
There exists a subsequence $\{\varphi^{\cF_{m_j}}_t\}_{j\in\NN}$ such that for any $p\geq1$ and $t\geq0$,
$$
\lim_{j\rightarrow0}d_p(\varphi^{\cF_{m_j}}_t,\varphi^\cF_t)=0.
$$
\end{proposition}

\begin{proof}
We need some ingredients and terminologies from the non-Archimedean approach developed in \cite{BBJ18,BHJ17,BoJ18,BoJ18a}. We refer the reader to \emph{loc. cit.} for more details. 

Let $\mathrm{FS}$ be the Fubini--Study map defined in \cite{BoJ18}. Set
$$
u:=\mathrm{FS}(||\cdot||_\bullet)\text{ and }u^m:=\mathrm{FS}(||\cdot||_{m,\bullet}),
$$
where $||\cdot||_\bullet$ denotes the graded norm associated to $\cF$ and $||\cdot||_{m,\bullet}$ the graded norm associated to $\cF_m$ (cf. \cite[Section 3]{BoJ18}).
Then $u\in \cE^{1,\mathrm{NA}}(L)$ and $u^m\in\cH^{\mathrm{NA}}(L)$ are functions on the Berkovich space $X^{\mathrm{an}}$. Moreover it follows from the construction that $u^m$ is an increasing net converging to $u$ (in the sense of \cite{BoJ18}). So by \cite[Theorem 6.11]{BoJ18a} one has
$$
E^{\mathrm{NA}}(u)=\lim_{m\rightarrow\infty}E^{\mathrm{NA}}(u^m),
$$
where $E^{\mathrm{NA}}$ denote the non-Archimedean Monge--Amp\`ere functional defined in \cite[Section 7]{BHJ17}. Note that by \cite[Theorem 6.6 and Example 6.9]{BBJ18}, $\varphi^\cF_t$ is the unique maximal geodesic ray emanating from $0$ that is associated to $u$ and likewise for each $m\geq 1$, $\varphi^{\cF_m}_t$ is the unique maximal geodesic ray emanating from $0$ that is associated to $u^m$. Moreover by \cite[Definition 6.5]{BBJ18}, $\{\varphi^{\cF_m}_t\}_{m\geq1}$ forms an increasing net. Now we show that there is an increasing subsequence converging pointwise to $\varphi^\cF_t$ for each $t\geq0$. To see this, we extract an increasing subsequence $\{\varphi^{\cF_{m_j}}_t\}_{j\geq1}$ and put
$$
\psi_t:=\lim_{j\rightarrow\infty}\varphi^{\cF_{m_j}}_t=\sup_{j}\varphi^{\cF_{m_j}}_t.
$$
So in particular,
$$
\varphi^{\cF_{m_j}}_t\leq\psi_t\leq\varphi^\cF_t\text{ for any }t\geq0 \text{ and }j\geq1.
$$
Observe that by Lemma \ref{lem:varphi-F-t-linearly-growth} and \cite[Lemma 3.28]{DarvasSurvey}, $\{\varphi_t^{\cF_{m_j}}\}_{j\geq1}$ is an increasing $d_p$-bounded sequence for any fixed $p\geq1$ and $t\geq0$. So by \cite[Lemma 3.34]{DarvasSurvey} we have $\psi_t\in\cE^p(X,\omega)$.
Meanwhile, applying \cite[Corollary 6.7]{BBJ18} we find that
$$
E(\varphi^\cF_t)/t=E^{\mathrm{NA}}(u)=\lim_{j\rightarrow\infty}E^{\mathrm{NA}}(u^{m_j})=\lim_{j\rightarrow\infty}E(\varphi^{\cF_{m_j}}_t)/t\text{ for any}\ t>0,
$$
where $E$ denotes the classical Monge--Amp\`ere functional (see \eqref{eq:def-E}).
So by \cite[Corollary 3.39]{DarvasSurvey} this forces that
$$
E(\psi_t)=E(\varphi^\cF_t)\text{ for any }t\geq0
$$
and hence (by \cite[Proposition 3.43]{DarvasSurvey}) $d_1(\psi_t,\varphi^\cF_t)=0$. Then from \cite[Proposition 3.27]{DarvasSurvey} we deduce that
$$
\psi_t=\varphi^\cF_t\text{ for any }t\geq0.
$$
Thus for any $t\geq0$, as $j\rightarrow\infty$,
$$
\varphi^{\cF_{m_j}}_t\text{ increasingly converges to }\varphi^\cF_t.
$$
Finally by \cite[Lemma 4.34]{DarvasSurvey} again, we conclude that
$$
d_p(\varphi^{\cF_{m_j}}_t,\varphi^\cF_t)\rightarrow 0\text{ as }j\rightarrow\infty,
$$
as desired.
\end{proof}

So the proof of Theorem \ref{thm:S-p=d_p/t-filtration} is complete.

Take $p=1$,
then Theorem \ref{thm:S-p=d_p/t-filtration} immediately implies (this should be well-known to experts):
\begin{equation}
\label{eq:S=E}
    S(L,\cF)=\frac{E(\varphi^\cF_t)}{t}\ \text{ for all $t>0$},
\end{equation}
where $E(\cdot)$ is the Monge--Amp\`ere functional defined by
\begin{equation}
    \label{eq:def-E}
    E(\varphi):=\frac{1}{(n+1)V}\int_X\sum_{i=0}^n\omega^{n-i}\wedge(\omega+dd^c\varphi)^i.
\end{equation}
This functional is known to be linear along $\varphi^\cF_t$.

Now assume that $F$ is a prime divisor over $X$. Combining Proposition \ref{prop:H-p-increase} with Theorem \ref{thm:S-p=d_p/t-filtration} we conclude that
$$
\bigg(\frac{n+p}{n}\bigg)^{1/p}\frac{d_p(0,\varphi^F_t)}{t}\text{ is non-decreasing in $p$}.
$$
It would be interesting to know if this holds for geodeisc rays induced by general valuations. But it is easy to see that
this monotonicity is in general not valid for geodesic rays induced by filtrations. For instance consider $\varphi_t:=Ct$ for some constant $C>0$, which is induced by the shifted trivial filtration: $\cF^{\lambda}R_m:=R_m$ for $\lambda\leq Cm$ and $\cF^\lambda R_m:=\{0\}$ for $\lambda>Cm$. Then clearly $\varphi_t$ violates the above monotonicity.
So geodesic rays constructed from valuative data typically enjoy additional properties. See also \cite{DX20} for related discussions in this direction.

\begin{remark}
As studied by Darvas--Lu \cite{DL19}, the $d_p$-distance between two geodesic rays yields rich information in Mabuchi geometry. In our context, we may consider two linearly bounded filtrations of $R$, say $\cF$ and $\cF^\prime$, then it is expected that
\begin{equation}
\label{eq:d-p-of-two-filtrations}
    d_p(\cF,\cF^\prime)=\lim_{t\rightarrow\infty}\frac{d_p(\varphi_t^\cF,\varphi_t^{\cF^\prime})}{t}.
\end{equation}
Here $d_p(\cF,\cF^\prime)$ denotes the $d_p$-distance between filtrations (see \cite{BoE} for the definition.).
When $\cF^\prime$ is trivial, \eqref{eq:d-p-of-two-filtrations} reduces to Theorem \ref{thm:S-p=d_p/t-filtration}. 
However, proving \eqref{eq:d-p-of-two-filtrations} in full generality seems to be a difficult problem.
\end{remark}




\bibliography{ref.bib}
\bibliographystyle{abbrv}

\end{document}